\documentclass{amsart}

\usepackage{amsmath}
\usepackage{verbatim}
\usepackage{amssymb}

\numberwithin{equation}{section}

\newtheorem{theorem}{Theorem}
\newtheorem{lemma}{Lemma}

\theoremstyle{definition}
\newtheorem{rem}{Remark}

\newcommand{\DD}{\mathrm D}
\newcommand{\cD}{\mathcal D}
\newcommand{\cC}{\mathcal C}

\newcommand{\cP}{\mathcal P}

\newcommand{\cK}{\mathcal K}
\newcommand{\cF}{\mathcal F}
\newcommand{\cT}{\mathcal T}
\newcommand{\IR}{\mathbb R}
\newcommand{\IN}{\mathbb N}
\newcommand{\tD}{\text D}
\newcommand{\tN}{\text N}
\newcommand{\vp}{\varphi}

\DeclareMathOperator{\Span}{span}
\DeclareMathOperator{\meas}{meas}
\DeclareMathOperator{\diag}{diag}

\title[Galerkin method for 
hyperbolic PDE with memory]
{Existence and convergence of Galerkin approximation 
for second order hyperbolic equations with memory term}

\author[F.~Saedpanah]{Fardin Saedpanah}

\address{Department of Mathematics, 
University of Kurdistan, P. O. Box 416, 
Sanandaj, Iran}

\email{f.saedpanah@uok.ac.ir\\
           f\_saedpanah@yahoo.com}



\begin{document}

\begin{abstract}
We study a second order hyperbolic initial-boundary value partial 
differential equation with memory, that results in an integro-differential 
equation with a convolution kernel. 
The kernel is assumed to be either smooth or no 
worse than weakly singular, that arise ,e.g., in linear and fractional 
order viscoelasticity. Existence and uniqueness of the spatial local 
and global Galerkin approximation of the problem is proved by means of 
Picard iteration. Then spatial finite element approximation of the problem 
is formulated, and optimal order a priori estimates are proved by energy method. 
The required regularity of the solution, for optimal order convergence, 
is the same as minimum regularity of the solution for second order 
hyperbolic partial differential equations.   
\end{abstract}

\date{January 27, 2014}

\maketitle
 
\section{Introduction}
We study, for any fixed $T>0$,  integro-differential 
equations of the form
\begin{equation}  \label{Problem}
  \ddot u+A u-\int_0^t K(t-s)Au(s)\,ds=f, \quad t \in (0,T), 
  \quad {\rm with} \ u(0)=u^0,\ \dot{u}(0)=u^1,
\end{equation}
(we use `$ \cdot $' to denote `$ \frac{\partial}{\partial t} $') 
together with a homogeneous Dirichlet boundary 
condition on a bounded polygonal domain $\Omega \subset \IR^d$, $d=1,2,3$ 
with boundary $\partial \Omega$, or with    
mixed homogeneous 
Dirichlet and nonhomogeneous Neumann boundary condition, that 
is important from practical view point. 
Here $A$ is a self-adjoint, positive definite, uniformly elliptic 
second order linear operator on a separable Hilbert space. 
The kernel $K$ is considered to be either smooth (exponential), 
or no worse than weakly singular, that is singular at the origin 
but locally integrable. 
We assume that the kernel has the properties
\begin{align}  \label{KernelProperty}
  K\geq 0,\quad \dot K \leq 0, \quad \| K\|_{L_1(0,T)}=\kappa<1.
\end{align}
Examples  of this type of problems, that appear ,e.g., in  
the theory of linear and fractional order viscoelasticity, is found in 
\cite{FardinEJM2014}, and references therein.

The Mittag-Leffler type kernels, as a chief example of fractional 
order kernels, is used in fractional order viscoelasticity, and  
interpolates between weakly singular kernels and smooth (exponential) 
kernels, that arise in the theory of linear viscoelasticity.   
This means that, the model problem \eqref{Problem} with the kernel satisfying 
\eqref{KernelProperty}, can capture the mechanical behavior for 
a wide class of materials, see \cite{FardinEJM2014} for more details and references.  
This is the reason for considering problem \eqref{Problem} with 
convolution kernel satisfying \eqref{KernelProperty}. 
Our  cheif examples for weakly dingular kernels are 
\begin{equation*}
  K(t)=\frac{t^{\alpha -1}}{\Gamma(\alpha)}, \quad 0<\alpha <1,
\end{equation*}
and the Mittag-Leffler type kernels, see \cite{FardinEJM2014} 
and \cite{StigFardin2010}.

There is an extensive literature on theoretical and numerical analysis 
for partial differential equations (PDEs) with memory, in particular 
integro-differential equations. To mention some, see  
\cite{FardinEJM2014}, \cite{StigFardin2010}, 
\cite{RiviereShawWhiteman2007}, 
\cite{AdolfssonEnelundLarsson2008},  
\cite{McLeanThomee2010}, 
\cite{FardinBIT2013}, 
\cite{PaniThomeeWahlbin1992}, 
\cite{AdolfssonEnelundLarsson2004},  
\cite{LinThomeeWahlbin1991},  
and their references. 

In \cite{FardinEJM2014} well-posedness of a problem, that is slightly 
more general than \eqref{Problem}, has been studied 
based on (a global) Galerkin approximation method. 
The first step, \cite[$\S 3.2$ ]{FardinEJM2014},  
is to prove existence of a unique solution of the approximate problem 
by Galerkin method, and using the Laplace tranform. 
In this work, we extend and give an alternative proof for 
existence and uniqueness of local and global Galerkin approximation 
methods, and we give a straightforward and constructive proof 
based on Picard iteration. 

We note that, we treat the model problem \eqref{Problem} as a 
hyperbolic second order PDE with memory. 
In \cite{Rauch}, using an explicit representation of the solution of 
the wave equation, it has been proved that the finite element 
approximation of the wave equation has optimal order of convergence, 
and an extra derivative of regularity of the solution is required 
to obtain this rate of convergence, and this regularity requirement is minimum. 

Spatial finite element approximation of hyperbolic integro-differential 
equations similar to \eqref{Problem} have been studied in 
\cite{LinThomeeWahlbin1991}, 
\cite{AdolfssonEnelundLarssonRacheva2006} and \cite{FardinBIMS2012}.   
For some fully discrete methods see \cite{FardinBIT2013} and 
references therein. 
In \cite{LinThomeeWahlbin1991} and 
\cite{AdolfssonEnelundLarssonRacheva2006}, for optimal order 
$L_\infty(L_2)$ a priori error estimate for the solution $u$, 
they require  two extra derivative of regularity of the solution. 
This was relaxed in \cite{StigFardin2010} and \cite{FardinBIMS2012} 
to one extra derivative, 
using the so called velocity-displacement form of the problem, 
that is a system of first order ordinary differential equations (ODEs) 
with respect to the time variable, 
and require stability estimates of a slightly more general problem. 
Here, using a different technique adapted from \cite{Baker1976}, for second 
order hyperbolic PDEs, we present an alternative   
proof to obtain $L_\infty(L_2)$ optimal order a priori 
error estimate with one extra derivative regularity of the solution $u$, 
similar to \cite{StigFardin2010} and \cite{FardinBIMS2012}. 
Here, we do not use the so called velocity-displacement formulation, and  
we give a short  and straightforward proof. Comparing with the second 
order hyperbolic PDEs, this one extra derivative seems to be minimal also 
for the counterpart integro-differential equations. 
However, a similar proof as in \cite{Rauch} 
can not be directly applied to our model problem \eqref{Problem}, 
due to the lack of an explicit representation of the solution. 
This minimal regularity assumption of the solution is also an important 
issue, e.g. in the error analysis of the finite element approximation of PDEs, 
see \cite{KovacsLarssonSaedpanah2010}. 
The present work also extend previous works, e.g., 
\cite{AdolfssonEnelundLarsson2008}, \cite{AdolfssonEnelundLarsson2004}, 
\cite{ShawWhiteman2004}, on quasi-static fractional order viscoelsticity 
$(\ddot{u}\approx 0)$ to the dynamic case.

The rest of the paper is  organized as follows. 
In $\S2$, we bring preliminaries and introduce the weak form of the problem. 
Then, in $\S3$ existence and uniqueness of local and global Galerkin 
approximation of the problem is proved.  
The spatial finite element discretization of the problem is formulated 
in $\S4$, and $L_\infty(L_2)$ and $L_\infty(H^1)$ optimal a priori error 
estimates for the displacement $u$ is proved together with 
$L_\infty(L_2)$ optimal a priori error estimate for the velocity $\dot u$. 

\section{Preliminaries and weak formulation}
We recall that $\Omega \subset \IR^d$, $d=1,2,3$, is a bounded 
polygonal domain with boundary $\partial \Omega$. 
We use  the standard Sobolev spaces $H^s=H^s(\Omega)^d$ with 
the corresponding norms $\|\cdot\|_s$ and inner products. 
We denote $H=H^0=L_2(\Omega)^d$ and $V=H_0^1(\Omega)^d$, 
and the corresponding norms by $\|\cdot\|$ and $\|\cdot\|_V$, respectively. 
We recall that $A$ is a self-adjoint, positive definite, uniformly elliptic 
second order linear operator on $\cD(A)=H^2 \cap V$. 
Therefore we can equip $V$ with the inner product 
$a(\cdot,\cdot)=(A\cdot,\cdot)$ and the corresponding norm 
$\|\cdot\|_V^2 = a(\cdot,\cdot)$, since $\|\cdot\|_v$ and $\|\cdot\|_1$ 
are equivalent. 

Then the weak form of the model problem \eqref{Problem} 
is read as, find $u(t) \in V$ such that
\begin{equation}   \label{weakform1}
\begin{split}
  &(\ddot u(t),v) + a(u(t),v)
     -\int_0^t\! K(t-s) a(u(s), v) \,ds \\
  &\qquad\qquad\qquad\qquad\qquad\quad
  = (f(t),v), \quad \forall v \in V ,\ t\in (0,T),\\
  &u(0)=u^0,\quad \dot u(0)=u^1.
\end{split}
\end{equation}

In the case of mixed homogeneous Dirichlet and 
nonhomogeneous Neumann boundary conditions, 
we assume that 
$\partial \Omega= \partial \Omega_{\tD} \cup \partial \Omega_{\tN}$, 
where $\partial \Omega_{\tD}$ and $\partial \Omega_{\tN}$ are disjoint and 
$\meas(\partial \Omega_{\tD}) \neq 0$. Then we denote 
$V=\{v \in H^1 : v|_{\partial \Omega_{\tD}} =0 \}$ 
and $H_{\partial \Omega_{\tN}}=L_2(\partial \Omega_{\tN})^d$. 
Our chief example for this case is $Au=-\nabla \sigma_0(u)$ 
where $\sigma_0$ is the standard stress tensor in elasticity, 
see e.g., \cite{StigFardin2010}. 
Then the corresponding nonhomogeneous Neumann boundary condition is 
\begin{equation*}
  \Big(\sigma_0 - \gamma \int_0^t K(t-s) \sigma_0(s) ds\Big) \cdot n =g(t),
\end{equation*}
and we introduce the bilinear form (with the usual summation convention)
\begin{equation*}
  a(u,v)=\int_{\Omega}\!\big(
  2\mu\epsilon_{ij}(u)\epsilon_{ij}(v)
   +\lambda\epsilon_{ii}(u)\epsilon_{jj}(v)\big)\,dx,
   \quad \forall u,v\in V,
\end{equation*}
where $\epsilon$ is the usual strain tensor and $\lambda, \mu$ are elastic 
constants of Lam\'e type. It is known that $a(\cdot,\cdot)$ is coercive. 
Then the weak form is read as, find $u(t) \in V$ such that
\begin{equation}   \label{weakform2}
\begin{split}
  &(\ddot u(t),v) + a(u(t),v)
     -\int_0^t\! K(t-s) a(u(s), v) \,ds \\
  &\qquad\qquad\qquad\qquad\qquad\quad
  = (f(t),v) + (g(t),v)_{\partial \Omega_{\tN}}, \quad \forall v \in V ,\ t\in (0,T),\\
  &u(0)=u^0,\quad \dot u(0)=u^1,
\end{split}
\end{equation}
where 
$(g(t),v)_{\partial \Omega_{\tN}}=\int_{\partial \Omega_{\tN}}\!g(t)\cdot v\,dS$.

\section{Existence and uniqueness of the Galerkin approximation solution}
To be complete, we consider the model problem with 
mixed homogeneous Dirichlet and nonhomogeneous Neumann 
boundary conditions, with weak form \eqref{weakform2}. 
In this section, we study existence and uniqueness of spatial approximate 
solution of \eqref{Problem} by local and global Galerkin 
methods, using Picard iteration. 
In particular we use the standard finite element Galerkin 
method (FEGM), as a local Galerkin method,  
and a global Galerkin method (GGM) 
based on the eigenfunctions of the operator $A$. 

Spatial semidiscretization of \eqref{Problem} can be based on a 
local  or global Galerkin method, using a finite dimensional subspace 
$V_h=\Span\{\vp_1,\dots,\vp_m\}$ of $V$. 
For the finite element method we use the standard finite element spaces $V_h$ 
consisting of continuous piecewise polynomials, corresponding to a triangulation 
of the computational domain $\Omega$. 
Let $\{(\lambda_j,\vp _j)\}_{j=1}^\infty$ be the eigenpairs 
of the weak eigenvalue problem
\begin{equation}   \label{weakeigenvalue}
  a(\vp,v)=\lambda(\vp,v),\quad \forall v\in V.
\end{equation}
It is known that $\{\vp _j\}_{j=1}^\infty$ can be chosen to be an ON-basis in 
$H$ and an orthogonal basis for $V$. 
Then any finite dimensional space 
$V_h=\Span\{\vp_1,\dots,\vp_m\}$ can be a basis for a GGM, 
so called spectral Galerkin method. 
Obviously limitations of the Galerkin methods, such as 
spectral Galerkin methods,  on the computational 
domain $\Omega$ should be considered. 

We recall  the $L_2$-projection $P_h:H\to V_h$ and 
the Ritz projection $R_h:V\to V_h$ defined by
\begin{equation} \label{PR_projections}
  a(R_h v, \chi)=a(v,\chi)\quad {\rm and}\quad 
  (P_h v,\chi)=(v,\chi),\qquad \forall \chi \in V_h.
\end{equation}
We also recall the truncated Fourier series projection 
$P_F:H \to V_h$, defined by
\begin{equation*}  
  P_F v = \sum_{j=1}^m (v,\vp_j)\vp_j = \sum_{j=1}^m \tilde v \vp_j .
\end{equation*}
 
Now, for a fixed positive integer $m\in\IN$, we seek a function of the form
\begin{equation}   \label{u_m}
  u_m(t)=\sum_{j=1}^m \alpha_j(t)\vp _j , 
\end{equation}
that satisfies 
\begin{equation}   \label{weakGalerkin}
  \begin{split}
    (\ddot u_m(t),\vp_k&) + a(u_m(t),\vp_k)
    -\int_0^t\! K(t-s) a(u_m(s), \vp_k)\,ds\\
    &= (f(t),\vp_k) + (g(t),\vp_k)_{\partial\Omega_\tN},
    \quad k=1,\,\dots,\,m ,\,t\in (0,T),
  \end{split}
\end{equation}
with initial conditions
\begin{equation}   \label{weakGalerkininitial}
  u_m(0)=P^0 u^0, \quad \dot u_m(0)=P^1 u^1.
\end{equation}
Here $P^0$ and  $P^1$ are suitable projections to be chosen, 
such as the standard $L_2$-projection, 
the Ritz projection, or the truncated Fourier 
series projection. 
We note that, denoting the vector $\alpha(t)=\big(\alpha_j(t) \big)_{j=1}^m$, 
the initial coefficients 
$\frac{d^l}{dt^l}\alpha(0) = \big(\frac{d^l}{dt^l}\alpha_j(0) \big)_{j=1}^m$, $l=0,1 $, are obtained from 
the system of linear equations, for $l=0,1$,
\begin{align} \label{initialSystem}
\begin{aligned}
  &M \frac{d^l}{dt^l}\alpha(0) = \big((u^l,\vp_i)\big)_{i=1}^m, 
    &&\textrm{ with $L_2$-projection},\\
  &S \frac{d^l}{dt^l}\alpha(0) = a\big((u^l,\vp_i)\big)_{i=1}^m, 
    &&\textrm{ with Ritz projection},\\
  &\frac{d^l}{dt^l}\alpha(0) = \big((u^l,\vp_i)\big)_{i=1}^m, 
    &&\textrm{ with truncated Fourier projection},
\end{aligned}
\end{align}
where $M$ is the standard mass matrix and  $S$ is the stiffness matrix. Therefore the system is uniquely solvable, 
and, e.g., for the last case we have $\frac{d^l}{dt^l}\alpha_j(0)=(u^l,\vp _j)$, that is, 
\begin{equation*}  
  \frac{d^l}{dt^l}u_m(0)=\sum_{j=1}^m(u^l,\vp _j)\vp _j, \quad l=0,1.
\end{equation*}

Now, we prove that there exists a unique solution of the form  
\eqref{u_m} for the spatial semidiscrete problem \eqref{weakGalerkin} 
with initial conditions \eqref{weakGalerkininitial}. 
\begin{lemma}
Assume that the initial data $u^0$ and $u^1$ are regular such that 
the initial conditions in \eqref{weakGalerkininitial} are well-defined, and 
$g\in L_1((0,T);H_{\partial\Omega_\tN} )$,  $f\in L_1((0,T);H)$. 
Then, for each $m\in \IN$, there exists a unique function 
$u_m$ of the form \eqref{u_m} satisfying 
\eqref{weakGalerkin}-\eqref{weakGalerkininitial}. 
\end{lemma}
\begin{proof}
We organize our proof in three steps. 
First, using a Galerkin approximation method, we formulate the 
spatial semidiscrete form of the main problem \eqref{Problem} 
as a system of second kind linear Volterra equations. 
Then we prove existence of the solution of the semidiscrete problem 
by means of Picard iteration method. Finally, we prove uniqueness of 
the solution.  

1. Substituting \eqref{u_m} in \eqref{weakGalerkin}, we have
\begin{equation}   \label{SystemweakGalerkin}
  \begin{split}
    \sum_{j=1}^m \Big\{ (\vp _j,\vp _k) \ddot \alpha_j(t)
    &+ a(\vp _j,\vp _k)  \alpha_j(t) 
    -  a(\vp _j,\vp _k) \int_0^t K(t-s)\alpha_j(s) ds \Big\} \\
    &= (f(t),\vp_k) + (g(t),\vp _k)_{\partial\Omega_\tN},
      \quad k=1,\,\dots,\,m ,\,t\in (0,T), 
  \end{split}
\end{equation}
that is a system of linear second order ODEs  
with the initial data \eqref{weakGalerkininitial}, that is computable from 
\eqref{initialSystem}. 
Recalling the vector notation $\alpha(t)=\big(\alpha_j(t)\big)_{j=1}^m$, 
we rewrite \eqref{SystemweakGalerkin} in the matrix form
\begin{equation}   \label{weakGalerkinmatrix1}
    M\ddot \alpha(t)
    + S \alpha(t)- S \int_0^t K(t-s)\alpha(s) ds =F(t)
         \quad t\in (0,T), 
\end{equation}
where 
\begin{equation*}   
  \begin{split}
    &M=(M_{k,j})_{k,j=1}^m=\big((\vp_j,\vp_k) \big)_{k,j=1}^m, \quad 
    S=(S_{k,j})_{k,j=1}^m=\big(a(\vp_j,\vp_k) \big)_{k,j=1}^m, \\
    &F(t)=\big(F_k(t)\big)_{k=1}^m
    =\big((f(t),\vp_k) + (g(t),\vp _k)_{\partial\Omega_\tN} \big)_{k=1}^m. 
  \end{split}
\end{equation*}
We note that, if we use FEGM, then $M$ and $S$ are the 
standard mass matrix and stiffness matrix, respectively. 
While, recalling the eigenvalue problem \eqref{weakeigenvalue}, 
if we use GGM we have 
$M=I_{m}$ and $S=\diag (\lambda_1,\dots,\lambda_m)$. 
Therefore, in both cases, $M$ and $S$ are nonsingular matrices. 

Now we write the system of linear second order ODEs 
\eqref{weakGalerkinmatrix1}, as a system of linear first order ODEs
\begin{equation}   \label{weakGalerkinmatrix2}
    \dot \DD(t)
    +\tilde M^{-1}\tilde  S \DD(t) 
      -\tilde M^{-1} \tilde{\tilde S} \int_0^t K(t-s)\DD(s) ds 
         =\tilde M^{-1}\tilde F(t),
         \quad t\in (0,T), 
\end{equation}
where
\begin{equation*}   
  \begin{split}
    &\DD(t)= \begin{bmatrix} \alpha(t) \\ \dot\alpha(t) \end{bmatrix}, \quad 
       \tilde F(t)=\begin{bmatrix} 0\\ F(t) \end{bmatrix}\\
    &\tilde M=\begin{bmatrix} M & 0\\ 
                                                    0 & M \end{bmatrix}, \quad
       \tilde S
           =\begin{bmatrix} 0& - M\\ S & 0 \end{bmatrix}, \quad 
       \tilde{\tilde S}
           =\begin{bmatrix} 0& 0\\ S & 0 \end{bmatrix}.
  \end{split}
\end{equation*}

Then integrating with respect to $t$ and interchanging the order of integrals 
for the convolution term, we have 
\begin{equation}   \label{integraleq}
  \begin{split}
     \DD(t)&=\int_0^t \tilde M^{-1} \Big(\tilde  S  
      - \int_s^t K(r-s) dr \ \tilde{\tilde S}  \Big)\DD(s)\ ds
               +\int_0^t \tilde M^{-1} \tilde F(s)\ ds + \DD(0)\\
              &=\int_0^t \cK(t,s) \DD(s)\ ds + \cF(t), \quad t\in (0,T),
  \end{split}         
\end{equation}
with obvious notations $\cK$ and $\cF$. 

2. Now we solve the above integral equation by Picard iteration, see e.g., 
\cite[\S 2.1]{Burton} or \cite[\S 8.2]{KelleyPeterson:Book}, 
that is defined by, for $t\in[0,T]$,
\begin{equation}  \label{Picarditeration}
  \begin{split}
     \DD^0(t) &= \cF(t),\\
     \DD^{n+1}(t)&=\int_0^t \cK(t,s) \DD^n(s)\ ds + \cF(t),\quad n=0,1,\dots.
  \end{split}         
\end{equation}
We need to show that the Picard iterates $\{\DD^n(t)\}_{n=0}^\infty$ 
converges uniformly on $[0,T]$. 
We note that 
\begin{equation*}  
     \DD^n(t)=\DD^0(t) + \sum_{n=0}^{n-1} \big( \DD^{n+1}(t) - \DD^n(t)\big),
\end{equation*}
that is the partial sum of the infinite series
\begin{equation}  \label{infiniteseries}
     \DD^0(t) + \sum_{n=0}^\infty \big( \DD^{n+1}(t) - \DD^n(t)\big).
\end{equation}
Therefore, we first prove that the inifinite series is convergent uniformly 
on $[0,T]$, that implies uniform convergence of the Picard iterates 
$\{\DD^n(t)\}_{n=0}^\infty$. 

Denoting the standard maximum vector and matrix norms, 
for $b \in \IR^m$ and $B \in \IR^{m\times m}$, by
\begin{equation*}  
  |b|_\infty = \max_{1\leq k\leq m}|b_k|,\quad
  |B|_\infty = \max_{1\leq k\leq m} \sum_{j=1}^m |B_{k,j}|,
\end{equation*}
we have, recalling $\cK$ and $\cF$ from \eqref{integraleq}, 
\begin{equation}  \label{boundZ}
   \sup_{0 \leq s\leq t \leq T} | \cK(t,s)|_\infty
   \leq |M^{-1}S|_\infty \big(1+\| K \|_{L_1(0,T)}\big)=:\textrm{Z},
\end{equation}
and, using the Cauchy-Schwarz inequality and the trace inequality, 
\begin{equation}  \label{boundZ0}
  \begin{split}
     \sup_{0 \leq t \leq T} | \DD^0(t)|_\infty
     &=\sup_{0 \leq t \leq T} | \cF(t)|_\infty
     \leq |\DD(0)|_\infty +  \int_0^t | \tilde F(s)|_\infty \ ds \\
     &\leq \max \big(|\alpha(0)|_\infty +|\dot \alpha(0)|_\infty  \big)\\
     &\quad + \max_{1\leq k\leq m}
      \Big(\int_0^t |(f(s),\vp_k)|+|(g(s),\vp _k)_{\partial\Omega_\tN}|\ ds\Big) \\
     &\leq \max \big(|\alpha(0)|_\infty +|\dot \alpha(0)|_\infty  \big)\\
     &\quad + \max_{1\leq k\leq m}
        \Big(\int_0^t \|f(s)\|\|\vp_k\| 
        + \|g(s)\|_{\partial\Omega_\tN}\|\vp_k\|_{\partial\Omega_\tN} \ ds\Big) \\
     &\leq \max \big(|\alpha(0)|_\infty +|\dot \alpha(0)|_\infty  \big)\\
     &\quad + \Big(\|f\|_{L_1((0,T);H)}
                   +  \|g\|_{L_1((0,T);H_{\partial\Omega_\tN})} \Big) \\
     &\qquad  \times \max_{1\leq k\leq m} 
                      \big(\|\vp_k\| + C_{\rm Trace}\|\vp_k\|_V \big) \\
     &=:\textrm{Z}^0.
  \end{split}
\end{equation}
Therefore, we have
\begin{equation*}  
    |\DD^1(t)-\DD^0(t)|_\infty 
    \leq \int_0^t |\cK(t,s)|_\infty |\DD^0(s)|_\infty \ ds 
    \leq \textrm{Z} \textrm{Z}^0 t ,
\end{equation*}

We need to show that, for $n= 0,1,\dots$,
\begin{equation}  \label{induction}
    |\DD^{n+1}(t)-\DD^n(t)|_\infty 
    \leq \frac{\textrm{Z}^{n+1}t^{n+1}}{(n+1)!} \textrm{Z}^0.
\end{equation}
The poof is by induction. Indeed, the case $n=0$ has already been proved above. 
If we assume that \eqref{induction} holds for $n-1$, then for $n$ we have
\begin{equation*}  
  \begin{split}
    |\DD^{n+1}(t)-\DD^n(t)|_\infty 
    &\leq \int_0^t |\cK(t,s)|_\infty |\DD^n(s)-\DD^{n-1}(s)|_\infty \ ds\\
    &\leq \textrm{Z} \int_0^t  \frac{\textrm{Z}^n s^n}{n!} \textrm{Z}^0 \ ds \\
    &\leq \frac{\textrm{Z}^{n+1}t^{n+1}}{(n+1)!} \textrm{Z}^0.
  \end{split}
\end{equation*}
Obviously, from \eqref{induction}, we have
\begin{equation*}  
    \sup_{t\in[0,T]}|\DD^{n+1}(t)-\DD^n(t)|_\infty 
    \leq \frac{\textrm{Z}^{n+1} T^{n+1}}{(n+1)!} \textrm{Z}^0,
\end{equation*}
that concludes the infinite series \eqref{infiniteseries} is 
uniformly convergent on $[0,T]$. 
Therefore the Picard iterates $\{\DD^n(t)\}_{n=0}^\infty$ converges uniformly on 
$[0,T]$, say to $\DD(t)$, i.e.,
\begin{equation*}  
    \DD(t)=\lim_{n\to \infty} \DD^n(t), \quad t \in [0,T].
\end{equation*}

3. Finally, it remains to prove uniqueness. 
Assume that $\tilde \DD$ be another solution of the integral 
equation \eqref{integraleq}. 
Then 
\begin{equation*}
  \DD(t) - \tilde \DD(t) 
     = \int_0^t \cK(t,s) \big(\DD(s) - \tilde \DD(s) \big) ds ,
\end{equation*}
and consequently, we have
\begin{equation*}
  |\DD(t) - \tilde \DD(t)|_\infty 
    \leq \int_0^t |\cK(t,s)|_\infty |\DD(s) - \tilde \DD(s) |_\infty ds 
      \leq \textrm{Z} \int_0^t  |\DD(s) - \tilde \DD(s) |_\infty ds, 
\end{equation*}
that, by Gronwall's inequality, implies that 
$|\DD(t) - \tilde \DD(t)|_\infty = 0$, and therefore $\DD=\tilde \DD$. 
Hence the uniqueness is proved. Now the proof is complete. 
\end{proof}

We note that Lemma 1 holds true for any Galerkin method for which 
matrices $M$ and $S$ in \eqref{weakGalerkinmatrix1} are well-defined 
and $M$ is invertible. 

\begin{rem} \label{RemarkPicardLemma}
We recall $\cK$ and $\cF$ from \eqref{integraleq} and the assumptions 
that the kernel $K$ and the load and surface terms $f$ and $g$ are integrable. 
Therefore $\cK$ and $\cF$ are continuous functions, and steps 2 and 
3 of the proof of Lemma 1 is concluded from \cite[Theorem 2.1.1]{Burton}. 
We also note that, using \eqref{Picarditeration} one can show by induction, 
that each $\DD^n \in \cC[0,T]$. 
Hence, taking the limit $n\to \infty$ of both sides of the Picard iteration 
\eqref{Picarditeration} and thanks to uniform convergence, 
we conclude that $\DD\in \cC[0,T]$ is a 
solution of the integral equation \eqref{integraleq}. 
However, we have added more details to the proof to be complete 
and to show how the initial data and the load and surface terms 
should be estimated. 
For example, let consider the global Galerkin method based on the 
eigenfuncations of the operator $A$, and recall that 
$M=I_m$, $S=\diag(\lambda_1, \dots , \lambda_m)$. 
Then recalling \eqref{initialSystem} we have, from \eqref{boundZ},
\begin{equation*}
  \textrm{Z} = (1+\kappa) \max_{1\leq k\leq m} |\lambda_k| ,
\end{equation*}
and, from \eqref{boundZ0}, 
\begin{equation*}
  \textrm{Z}^0 = \|u^0\| + \|u^1\| 
     + \|f\|_{L_1((0,T);H)} 
       + C_{\textrm{Trace}} \|g\|_{L_1((0,T);H_{\partial\Omega_{\tN}})}
          \max_{1\leq k\leq m} \|\vp_k\|_V ,
\end{equation*}
and we note that for this case we need to assume $u^0 \in H$ and $u^1 \in H$. 
\end{rem}

Further analysis on regularity estimates of the solution can be found in 
\cite{FardinEJM2014}. 
\section{The spatial finite elment discretization}
In this section, for simplicity, we consider pure homogeneous Dirichlet 
boundary condition. 
We recall the variational form \eqref{weakform1}.

Let $\Omega$ be a convex polygonal domain and 
$\{\cT_h\}$ be a regular family of 
 triangulations of $\Omega$ with corresponding family of 
 finite element spaces $V_h^l \subset V$,
 consisting of continuous piecewise polynomials of degree 
 at most $l-1$, that vanish on 
$\partial\Omega$ (so the mesh is required to fit $\partial \Omega$). 
Here $l\ge 2$ is an integer number. 
We define piecewise constant mesh function 
$h_\mathsf{K}(x) = \textrm{diam}(\mathsf{K})$ for 
$x\in \mathsf{K},\, \mathsf{K}\in \mathcal{T}_h$, 
and for our error analysis we denote 
$h=\max_{\mathsf{K}\in\mathcal{T}_h}h_\mathsf{K}$. 
We note that the finite element spaces $V_h^l$ have the property that
\begin{equation} \label{GeneralErrorEstimate}
  \min_{\chi \in V_h^l}
  \{\|v-\chi\| + h\|v-\chi\|_1\}\le C h^i\|v\|_i, 
  \quad \textrm{for} \ v \in  H^i \cap V,\ 1\le i\le l.
\end{equation} 

We recall  the $L_2$-projection $P_h$ and 
the Ritz projection $R_h$ from \eqref{PR_projections}. 
We also recall the elliptic regularity estimate 
\begin{equation*}
  \|u\|_2 \leq C \|Au\|,\quad u \in \cD(A) = H^2 \cap V,
\end{equation*}
such that the error estimates \eqref{GeneralErrorEstimate} hold true 
for the Ritz projection $R_h$, see \cite{Thomee_Book}, i.e.,
\begin{equation} \label{errorRh}
  \begin{split}
    \|(R_h-I)v\| 
    + h\|(& R_h-I)v\|_1\le C h^i\|v\|_i, 
    \quad \textrm{for} \ v \in H^i \cap V,\ 1\le i\le l.
  \end{split}
\end{equation} 

Then, the spatial finite element discretization of \eqref{weakform1} 
is to find $u_h(t) \in V_h^l$ such that 
$u_h(0)=u_h^0,\ \dot u_h(0)=u_h^1$, and for $t\in(0,T)$,
\begin{equation}  \label{FE}
  (\ddot u_h,v_h)+a(u_h,v_h) -\int_0^t\! K(t-s)a(u_h(s),v_h)\ d s=(f,v_h),  
    \quad \forall v_h \in V_h^l, 
\end{equation}
where $u_h^0$ and $u_h^1$ are suitable approximations to be chosen, 
respectively, for $u_0$ and $u^1$ in $V_h^l$.

For our analysis, we define a function $\xi$ by 
\begin{align}   \label{xi}
  \xi(t)=\kappa-\int_0^t\! K(s)\,ds
    =\int_t^T\! K(s)\,ds, \quad t \in [0,T],
\end{align}
and, having \eqref{KernelProperty}, it is easy to see that
\begin{align}   \label{xiproperty}
  \begin{split}
    D_t\xi(t)=- K(t)<0,\quad
    \xi(0)=\kappa,\quad \xi(T)=0,\quad 
    0\leq \xi(t) \leq \kappa.
  \end{split}
\end{align}
Hence, $\xi$ is a completely monotone function, since
\begin{align*}
  (-1)^jD_t^j\xi(t)\ge 0,\quad t\in (0,T),\,j=0,1,2,
\end{align*}
and consequently $\xi\in L_{1,loc}[0,\infty)$ is a positive type kernel, 
that is, for any $T\ge 0$ and $\phi\in \cC([0,T])$,
\begin{align}   \label{positivetype}
  \int_0^T\!\int_0^t\!\xi(t-s)\phi(t)\phi(s)\,ds\,dt\ge 0 .
\end{align}

\begin{theorem} \label{Apriori-Semidiscrete}
Assume that $\Omega$ is a convex polygonal domain. 
Let $u$ and $u_h$ be, respectively, the solutions of 
\eqref{weakform1} and \eqref{FE}.
Then  

$(i)$ with initial condition $u_h^1=P_h u^1$, we have
\begin{equation}   \label{A_Priori_1}
  \begin{split}
     \|u_h(T)-u(T)\|
     &\le C\|u_h^0-R_h u^0\|
     + Ch^l\Big(\|u(T)\|_l +\int_0^T\!\|\dot u\|_l \,d\tau\Big),
  \end{split}
\end{equation}

$(ii)$ and with initial condition $u_h^0=R_h u^0$, we have
\begin{equation}   \label{A_Priori_2}
    \|\dot u_h(T)-\dot u(T)\|
    \le C\|u_h^1-R_h u^1\|
      + C h^l\Big(\|\dot u(T)\|_l +\int_0^T\! \| \ddot u\|_l\, d\tau\Big).
\end{equation}
\begin{equation}   \label{A_Priori_3}
    \|u_h(T)-u(T)\|_V
    \le C\|u_h^1-R_h u^1\|_V
      + C h^{l-1}\Big(\|u(T)\|_l +\int_0^T\! \| \ddot u\|_{l-1}\, d\tau\Big).
\end{equation}
\end{theorem}

\begin{proof}
The proof is adapted from \cite{Baker1976}. 
We split  the error as 
\begin{equation}\label{Error_Split}
  e=u_h-u=(u_h-R_hu)+(R_hu-u)=\theta+\omega.
\end{equation}
We need to estimate $\theta$, since the spatial projection error $\omega$ 
is estimated from \eqref{errorRh}. 

So, putting $\theta$  in  \eqref{FE} we have, for $v_h \in V_h^l$, 
\begin{equation*} 
  \begin{split}
    (\ddot{\theta},v_h&)+a(\theta,v_h)-\int_0^t K(t-s) a(\theta(s),v_h)ds\\
    &= (\ddot u_h,v_h)+a(u_h,v_h)-\int_0^t K(t-s) a(u_h(s),v_h)ds\\
    &\quad - (R_h\ddot{u},v_h)-a(R_h u,v_h)+\int_0^t K(t-s) a(R_hu(s),v_h)ds,
  \end{split}
\end{equation*}
that using \eqref{FE}, the definition of the Ritz projection $R_h$, and 
\eqref{weakform1}, we have 
\begin{equation}\label{theta_eq1}
  \begin{split}
    (\ddot{\theta},v_h&)+a(\theta,v_h)-\int_0^t K(t-s) a(\theta(s),v_h)ds\\
    &= (f,v_h)-(R_h\ddot{u},v_h)-a(u,v_h)+\int_0^t K(t-s) a(u(s),v_h)ds\\
    &= (\ddot{u},v_h)-(R_h\ddot{u},v_h)
    =-(\ddot{\omega},v_h).
  \end{split}
\end{equation}
Therefore we can write, for $v_h(t) \in V_h^l$, $t \in (0,T]$,
\begin{equation*}
  \frac{d}{dt}(\dot{\theta},v_h)-(\dot{\theta},\dot v_h)
    +a(\theta,v_h)- \int_0^t K(t-s) a(\theta(s),v_h(t))ds
   =-\frac{d}{dt}(\dot{\omega},v_h)+(\dot{\omega},\dot v_h),
\end{equation*}
that, recalling $e=\theta+\omega$, we obtain
\begin{equation}\label{Psi_eq2}
  -(\dot{\theta},\dot v_h)+a(\theta, v_h)- \int_0^t  K(t-s) a(\theta(s),v_h(t))ds
  =-\frac{d}{dt}(\dot{e},v_h)+(\dot{\omega},\dot v_h)
\end{equation}
Now let $ 0<\varepsilon \leq T$ , and we make the particular choice
\begin{equation*}
  v_h(\cdot,t)
  =\int_t^{\varepsilon}\theta(\cdot,\tau)d\tau ,\qquad  0\leq t \leq T,
\end{equation*}
then clearly we have
\begin{equation} \label{zz}
  v_h(\cdot,\varepsilon)=0,\quad \frac{d}{dt}v_h(\cdot,t)=-\theta(\cdot,t),
   \quad 0\leq t \leq T.
\end{equation}
Hence, considering \eqref{zz} in \eqref{Psi_eq2} , we have
\begin{equation*}
\frac{1}{2} \frac{d}{dt} (\|\theta \|^2-\|v_h \|_V^2)
-\int_0^t K(t-s) a(\theta(s),v_h(t))ds
=-\frac{d}{dt}(\dot{e},v_h)-(\dot{\omega},\theta).
\end{equation*} 
Now, integrating from $t=0$ to $t=\varepsilon$, we have
\begin{equation*}
\begin{split}
   \|\theta(\varepsilon) \|^2-\|\theta(0) &\|^2-\|v_h(\varepsilon)\|^2_V
       + \|v_h(0)\|^2_V-2\int_0^{\varepsilon}\int_0^t K(t-s) a(\theta(s),v_h(t))dsdt\\
   &=-2(\dot{e}(\varepsilon),v_h(\varepsilon))+ 2(\dot{e}(0),v_h(0))
       -2\int_0^{\varepsilon}(\dot \omega,\theta)dt.
\end{split}
\end{equation*}
Then, using the initial assuption $u_h^1=\cP_h u^1$  that implies  
the second term on the right side is zero and recalling 
$v_h(\varepsilon)=0$, we conclude
\begin{equation*}
\begin{split}
   \|\theta(\varepsilon) \|^2 + \|v_h(0)\|^2_V  &
     - 2\int_0^{\varepsilon}\int_0^t K(t-s) a(\theta(s),v_h(t))dsdt\\
  &=\|\theta(0) \|^2 - 2\int_0^{\varepsilon}(\dot \omega,\theta)dt,
\end{split}
\end{equation*}
that, using the Cauchy-Schwarz inequality, implies
\begin{equation}  \label{eq3}
  \begin{split}
  \|\theta(\varepsilon) \|^2 +\|&v_h(0)\|_V^2 - 2\int_0^{\varepsilon}
     \int_0^t K(t-s) a(\theta(s),v_h(t))dsdt \\
  &\leq \|\theta(0) \|^2+2\max_{0 \leq t \leq \varepsilon} \|\theta (t)\| 
     \int _0^{\varepsilon} \|\dot \omega\|dt.
  \end{split}
\end{equation}

Now, by changing the order of integrals, 
using $ \frac{d}{dt}\xi(t-s)=-K(t-s) $ from \eqref{xiproperty}, 
and integration by parts, we can write the third term on the left 
side as
\begin{equation*}
\begin{split}
  -2\int_0^{\varepsilon}\int_0^t &K(t-s) a(\theta(s),v_h(t))dsdt
  =2\int_0^{\varepsilon}\int_s^{\varepsilon}\frac{d}{dt}\xi(t-s) a(\theta(s),v_h(t))dtds\\
  &=2\int_0^{\varepsilon} \xi(\varepsilon -s) 
   a(\theta(s),v_h(\varepsilon))ds-2\int_0^{\varepsilon} \xi(0) a(\theta(s),v_h(s))ds\\
  &\qquad-2\int_0^{\varepsilon}\int_s^{\varepsilon}\xi(t-s)a(\theta(s),\dot v_h(t))dtds.
\end{split}
\end{equation*}
Then, using \eqref{zz} and $\xi(0)=\kappa$, we have
\begin{equation*}
  \begin{split}
    -2\int_0^{\varepsilon}&\int_0^t K(t-s) a(\theta(s),v_h(t))dsdt\\
    &=\kappa(\|v_h(\varepsilon)\|^2_V-\|v_h(0)\|^2_V)
   +2\int_0^{\varepsilon}\int_s^{\varepsilon} \xi(t-s) a(\theta(s),\theta(t))dtds.
  \end{split}
\end{equation*}
Therefore, using this and $v_h(\varepsilon)=0$ in \eqref{eq3} we have
\begin{equation*}  
  \begin{split}
  \|\theta(\varepsilon) \|^2 +(1-\kappa)\|&v_h(0)\|_V^2 
  +2\int_0^{\varepsilon}\int_s^{\varepsilon} \xi(t-s) a(\theta(s),\theta(t))dtds\\
  &\leq \|\theta(0) \|^2+2\max_{0 \leq t \leq \varepsilon} \|\theta (t)\| 
     \int _0^{\varepsilon} \|\dot \omega\|dt,
  \end{split}
\end{equation*}
that considering the fact that $\xi$ is a positive type kernel and $\kappa<1$, 
in a standard way, implies that
\begin{equation*}
   \|\theta(T)\|\leq C(\|\theta(0)\|+\int_0^T\|\dot{\omega}\|d\tau). 
\end{equation*}
Hence, recaling \eqref{Error_Split}, we have
\begin{equation*} 
  \begin{split}
   \|e(T)\|
   &\leq \|\theta(T)\|+\|\omega(T)\|\\
   &\leq  C\Big(\|u_h^0-R_hu^0\|+\int _0^T\| R_h \dot u- \dot u\|dt\Big)
     +\|(R_hu -u)(T)\|,
  \end{split}
\end{equation*} 
that using the error estimate \eqref{errorRh} implies the a priori error 
estimate \eqref{A_Priori_1}. 

 Now, to prove the second and the third error estimates 
 \eqref{A_Priori_2}-\eqref{A_Priori_3}, we choose 
 $v_h=\dot \theta(t)$ in \eqref{theta_eq1}.  
 Then we have 
\begin{equation}\label{z_{3}}
\begin{split}
   \|\dot{\theta}(t)\|^2+\|\theta(t)\|_V^2
      -2\int_0^t\int_0^{\tau}&K(\tau-s) 
          a(\theta (s),\dot{\theta}(\tau))\ ds d\tau  \\
  &=\|\dot{\theta}(0) \|^{2}+\|\theta(0)\|_V^{2}
      -2\int_0^t(\ddot{\omega},\dot{\theta})d\tau. 
\end{split}
\end{equation}
We can write the third term in the left side as, recalling
  $K(t-s)=\frac{d}{ds}\xi(t-s)$, $\xi(0)=\kappa$ from \eqref{xiproperty}, 
  and integration by parts
\begin{equation}\label{z_{4}}
\begin{split} 
    -2\int_0^t\int_0^{\tau}&K(\tau-s) a(\theta (s),\dot{\theta}(\tau))\ dsd\tau \\
    &= -2\int_0^t\int_0^{\tau} \frac{d}{ds} 
          \xi(\tau-s) a(\theta (s),\dot{\theta}  (\tau))\ dsd\tau \\
    &=-2\int_0^t \xi(0) a(\theta (\tau),\dot{\theta}(\tau))\ d\tau
         + 2\int_0^t \xi(\tau) a(\theta (0),\dot{\theta}(\tau))\ d\tau\\
    &\qquad + 2\int_0^t\int_0^{\tau} \xi(\tau-s) 
          a(\theta' (s),\dot{\theta}(\tau))\ dsd\tau\\
    &=-\kappa\|\theta(t)\|_V^2+\kappa\|\theta(0)\|_V^2
          + 2\int_0^t \xi(\tau) a(\theta (0),\dot{\theta}(\tau))\ d\tau\\
    &\qquad +2\int_0^t\int_0^{\tau} \xi(\tau-s) 
               a(\theta' (s),\dot{\theta}(\tau))\ dsd\tau.
\end{split}
\end{equation}
Then, putting \eqref{z_{4}} in \eqref{z_{3}} and using  the initial value assumption 
$\theta(0)=u_h^0-R_h u^0=0$, we have
\begin{equation*} 
\begin{split}
   \|\dot{\theta}(t)\|^2+(1-\kappa)\|\theta(t)\|_V^2
      +2\int_0^t\int_0^{\tau}& \xi(\tau-s) 
          a(\theta' (s),\dot{\theta}(\tau))\ ds d\tau  \\
  &=\|\dot{\theta}(0) \|^{2}
      -2\int_0^t(\ddot{\omega},\dot{\theta})d\tau. 
\end{split}
\end{equation*}
That, using the fact that $\xi$ is a positive type kernel, 
in a standard way, we have
\begin{equation*}
  \|\dot{\theta}(t)\| +(1-\kappa)\|\theta(t)\|_V^2 
   \leq C\big(\|\dot{\theta}(0) \|+\int_0^t\|\ddot{\omega}\|d\tau \big).
\end{equation*}
Hence, recalling \eqref{Error_Split} and $t\in(0,T]$, we have
\begin{equation*}
\begin{split}
  \|\dot{e}(T)\|&\leq \|\dot{\theta}(T)\|+\|\dot{\omega}(T)\|\\
  &\leq C\Big(\|u_h^1-R_hu^1\|
     +\int _0^T\|(R_{h}\ddot {u}-\ddot {u})(\tau)\ \|d\tau \Big)
       +\|(R_{h}\dot {u}-\dot {u})(T)\|\\
  \|e(T)\|_V&\leq \|\theta(T)\|_V+\|\omega(T)\|_V\\
  &\leq C\Big(\|u_h^1-R_hu^1\|
     +\int _0^T\|(R_{h}\ddot {u}-\ddot {u})(\tau) \| \ d\tau \Big)
       +\|(R_hu-u)(T)\|_V. 
\end{split}
\end{equation*}
These, using the error estimate \eqref{errorRh}, imply the a priori error 
estimates \eqref{A_Priori_2} and \eqref{A_Priori_3}. 
The proof is now complete. 
\end{proof}

We note that important tools in our error analysis is properly choosing initial 
data approximations and using 
the auxiliary function $\xi$, that is of positive type, 
and integration by parts, that simplifies the proof. 

Illustration of the order of convergence by our numerical examples are similar 
to, e.g.  \cite{StigFardin2010}, and therefore we have not presented 
numerical experiments, for short.




\begin{thebibliography}{9}
\bibitem{FardinEJM2014}
F.~Saedpanah,  
Well-posedness of an integro-differential equation with positive type 
kernels modeling fractional order viscoelasticity, 
European J Mech-A Solid 44, (2014) 201--211.
\bibitem{StigFardin2010}
S.~Larsson and F.~Saedpanah,  
The continuous {G}alerkin method for an integro-differential equation 
modeling dynamic fractional order viscoelasticity. 
IMA J Numer Anal 30, (2010), 964--986.
\bibitem{RiviereShawWhiteman2007}
B.~Rivi\`ere, S.~Shaw, and J.~R.~Whiteman,  
Discontinuous {G}alerkin finite element methods for dynamic 
linear solid viscoelasticity problems, 
Numer Methods Partial Differential Equations 23, (2007), 1149--1166.
\bibitem{AdolfssonEnelundLarsson2008}
K.~Adolfsson, M.~Enelund, and S.~Larsson,  
Space-time discretization of an integro-differential equation 
modeling quasi-static fractional-order viscoelasticity,  
J Vib Control 14, (2008), 1631--1649.
\bibitem{McLeanThomee2010}
W.~McLean and V.~Thom\'ee,  
Numerical solution via {L}aplace transforms of a fractional order 
evolution equation, 
J Integral Equations Appl 22, (2010), 57--94.
\bibitem{FardinBIT2013}
F.~Saedpanah,  
A posteriori error analysis for a continuous space-time finite element 
method for a hyperbolic integro-differential equation,  
BIT Numer Math 53, (2013) 689--716. 
\bibitem{PaniThomeeWahlbin1992} 
A.~K.~Pani and V.~Thom\'ee and  L.~B.~Wahlbin, 
Numerical methods for hyperbolic and parabolic integro-differential 
equations, 
J Integral Equations Appl 4, (1992) 533--584.
\bibitem{AdolfssonEnelundLarsson2004}
K.~Adolfsson, M.~Enelund, and S.~Larsson,    
Adaptive discretization of fractional order viscoelasticity 
using sparse time history,  
Comput Methods Appl Mech Engrg  193, (2004), 4567--4590.
\bibitem{LinThomeeWahlbin1991} 
Y.~Lin, V.~Thom\'ee, and L.~B.~Wahlbin, 
Ritz-Volterra projections to finite-element spaces and application 
to integro-differential and related equations, 
SIAM J Numer Anal 28, (1991), 1047-1070. 
\bibitem{Rauch}
J.~Rauch, 
On convergence of the finite element method for the wave equation, 
SIAM J Numer Anal 22, (1985) 245--249.
\bibitem{AdolfssonEnelundLarssonRacheva2006}
K.~Adolfsson, M.~Enelund, S.~Larsson,  and M.~Racheva,   
Discretization of integro-differential equations modeling  
dynamic fractional order viscoelasticity, 
LNCS 3743, (2006), 76--83.
\bibitem{FardinBIMS2012} 
F.~Saedpanah, 
Optimal order finite element approximation for a hyperbolic   
integro-differential equation, 
BIMS 38, (2012), 447--459. 
\bibitem{Baker1976} 
G.~A.~Baker, 
Error estimates for finite element methods for second order 
hyperbolic equations, 
SIAM J Numer Anal 13, (1976), 564--576.
\bibitem{KovacsLarssonSaedpanah2010} 
M.~Kov\'acs and S.~Larsson and F.~Saedpanah, 
Finite element approximation for the linear stochastic 
wave equation with additive noise, 
SIAM J Numer Anal 48, (2010) 408--427.
\bibitem{ShawWhiteman2004}
S.~Shaw and J.~R.~Whiteman,  
A posteriori error estimates for space-time finite element approximation 
of quasistatic hereditary linear viscoelasticity problems,  
Comput Methods Appl Mech Engrg 193, (2004), 5551--5572. 
\bibitem{Burton} 
T.~A.~Burton, 
Volterra {I}ntegral and {D}ifferential {E}quations, 
Mathematics in Science and Engineering 202, 
Elsevier, Berlin, 2005. 
\bibitem{KelleyPeterson:Book} 
W.~G.~Kelley and A.~C.~Peterson, 
The {T}heory of {D}ifferential {E}quations, 
Springer, New York, 2010.
\bibitem{Thomee_Book}
V.~Thom{\'e}e, 
Galerkin {F}inite {E}lement {M}ethods for {P}arabolic {P}roblems, 
Springer-Verlag, Berlin, 2006.
\end{thebibliography}


\end{document}